\documentclass[twoside, 11pt]{article}
\usepackage[top=1in,bottom=1in,left=1in,right=1in,footnotesep=0.5in]{geometry}

\usepackage[utf8]{inputenc}
\usepackage{MnSymbol}
\usepackage{booktabs} 
\usepackage{enumitem}
\usepackage{amsfonts, amsmath, amsthm}

\usepackage{color}
\definecolor{darkred}{RGB}{150,0,0}
\definecolor{darkgreen}{RGB}{0,150,0}
\definecolor{darkblue}{RGB}{0,0,150}

\usepackage{hyperref}
\hypersetup{colorlinks=true, linkcolor=darkred, citecolor=darkgreen, urlcolor=darkblue}
\usepackage{url}

\usepackage[sort&compress]{natbib} 
\def\HD{{\sf H}}

\theoremstyle{remark}

\def\beq{\begin{equation}} 
\def\eeq{\end{equation}}
\def\beqn{\begin{eqnarray*}}
\def\eeqn{\end{eqnarray*}}
\def\Bitem{\begin{itemize}\setlength{\itemsep}{.2in}}
\def\bitem{\begin{itemize}\setlength{\itemsep}{.05in}}
\def\eitem{\end{itemize}}
\def\Benum{\begin{enumerate}\setlength{\itemsep}{.2in}}
\def\benum{\begin{enumerate}\setlength{\itemsep}{.05in}}
\def\eenum{\end{enumerate}}
\def\bmult{\begin{multline*}}
\def\emult{\end{multline*}}
\def\bcenter{\begin{center}}
\def\ecenter{\end{center}}
\def\bframe{\begin{frame}}
\def\eframe{\end{frame}}

\newcommand{\secref}[1]{Section~\ref{sec:#1}}

\DeclareMathOperator*{\argmin}{arg\, min}
\DeclareMathOperator{\diam}{diam}
\DeclareMathOperator{\dist}{dist}

\DeclareMathOperator{\supp}{supp}

\def\cF{\mathcal{F}}

\def\cX{\mathcal{X}}

\renewcommand{\P}{\operatorname{\mathbb{P}}}

\def\eps{\varepsilon}
\def\implies{\ \Rightarrow \ }
\def\iff{\ \Leftrightarrow \ }

\def\1{\mathbbm{1}}

\def\dist{\mathsf{d}}

{
  \theoremstyle{plain}
  \newtheorem{lemma}{Lemma}
}
{
  \theoremstyle{plain}
  \newtheorem{theorem}{Theorem}
}
{
  \theoremstyle{plain}
  
}
{
  \theoremstyle{plain}
  \newtheorem{remark}{Remark}
}

\begin{document}

\title{On the Consistency of Metric and Non-Metric $K$-medoids}

\author{Ery Arias-Castro \and He Jiang}

\date{University of California, San Diego}

\maketitle

\begin{abstract}
We establish the consistency of $K$-medoids in the context of metric spaces. We start by proving that $K$-medoids is asymptotically equivalent to $K$-means restricted to the support of the underlying distribution under general conditions, including a wide selection of loss functions. This asymptotic equivalence, in turn, enables us to apply the work of \cite{parna1986strong} on the consistency of $K$-means. This general approach applies also to non-metric settings where only an ordering of the dissimilarities is available. We consider two types of ordinal information: one where all quadruple comparisons are available; and one where only triple comparisons are available. We provide some numerical experiments to illustrate our theory. 
\end{abstract}

\section{Introduction}  \label{sec:intro}

Cluster analysis is widely regarded as one of the most important tasks in unsupervised data analysis \citep{jain1999data, kaufman2009finding}. In this paper, we consider several center based clustering methods. Specifically, we show the asymptotic equivalence of $K$-means and $K$-medoids, and use this equivalence to prove the consistency of $K$-medoids in metric and non-metric (i.e., ordinal) settings.

\subsection{$K$-means and $K$-medoids}
\label{subsec:1_1}

The problem of $K$-means can be traced back to the 1960's to early work of \cite{macqueen1967some}. As the problem is computationally difficult in higher dimensions or when the number of clusters is large, it is instead most often approached via iterative methods such as Lloyd's algorithm \citep{lloyd1982least}. Leaving these computational challenges aside, assuming the problem is solved exactly, the consistency of $K$-means as a method has been thoroughly addressed in the literature. Early in this line of work, \cite{pollard1981strong} established the consistency of $K$-means in Euclidean spaces. \cite{parna1986strong} extended the result to separable metric spaces, while \cite{parna1988stability, parna1990existence, parna1992clustering} examined the particular situation of Hilbert and Banach spaces, where the existence of an optimal solution had been considered by \cite{herrndorf1983approximation} and \cite{cuesta1988strong}.

The problem of $K$-medoids dates back to the 1980's to work of \cite{kaufmann1987clustering}, who in the process proposed the Partition Around Medoids (PAM) iterative algorithm. 
\cite{van2003new} discovered that the original PAM has problem with recognizing rather small clusters, and defined a new version of PAM based on maximizing average silhouette, as defined by \cite{kaufman725pj}. Later \cite{park2009simple} proposed a computationally simpler version of PAM akin to Lloyd's algorithm for $K$-means. See \citep[Ch 2]{kaufman2009finding}. 
In a setting where the goal is the clustering of data sequences, \cite{wang2019k} established an exponential consistency result for $K$-medoids itself (when solved exactly).
To the best of our knowledge, however, the consistency of $K$-medoids in the more standard setting of clustering points in a metric space has not been previously established. 

\begin{quote}
{\em We establish the consistency of $K$-medoids by first showing that $K$-medoids is asymptotically equivalent to $K$-means restricted to the support of the underlying distribution, and then leveraging the work of \cite{parna1986strong} on the consistency of $K$-means in metric spaces.}
\end{quote}


\subsection{Ordinal $K$-medoids}
\label{subsec:1_2}

Beyond the more standard setting where the distances are available to us, we also consider ordinal settings where only an ordering of the distances is available. 
Even when the dissimilarities are available, turning them into ranks, and thus only working with the underlying ordinal information, can be attractive in situations where the numerical value of the dissimilarities has little meaning besides providing an ordering. This is the case, for example, in psychological experiments where human subjects are tasked with rating some items in order of preference.
Working with ranks also has the advantage of added robustness to outliers.

Statisticians and other data scientists have dealt with ordinal information for decades. Without going too far afield into rank-based inference \citep{hajek1967theory} or ranking models \citep{bradley1952rank}, there is non-metric scaling, aka ordinal embedding, which is the problem of embedding a set of items based on an ordering of their pairwise dissimilarities, with pioneering work in the 1960's by \cite{shepard1962i,shepard1962ii} and \cite{kruskal1964multidimensional}. The consistency of ordinal embedding --- by which we mean any solution to the problem assuming one exists --- was already considered by \cite{shepard1966metric}, and more thoroughly addressed only recently by \cite{kleindessner2014uniqueness} and \cite{arias2017some}.

Even closer to our situation, in the area of clustering, we know that hierarchical clustering with either single or complete linkage (or the less popular median linkage) only use the ordinal information, as can be seen from the fact that the output grouping remains the same if the dissimilarities are transformed by the application of a monotonically increasing function. The well-known clustering method DBSCAN of \cite{ester1996density} can be seen as a robust variant of single linkage, in its nearest-neighbor formulation, only relies on ordinal information as well.
On the other hand, hierarchical clustering with either average linkage or Ward's criterion does not have that property.
The use of $K$-medoids in ordinal settings does not seem nearly as widespread. In fact, we could only find a few references where the idea was proposed, scattered across various fields such as computer vision \citep{zhu2011rank} and data mining \citep{zadegan2013ranked}. 
In the context of an application to the clustering of pictures of human faces, \cite{zhu2011rank} proposed a rank order distance (ROD) based on a sum of individual ranks, acquired from triple comparisons, and then applied single linkage hierarchical clustering with this distance. They argued that this distance was more appropriate for their particular application than the more standard $L_1$ distance. In a followup work, \cite{huang2020adaptive} proposed a kernel variant of ROD. With the intention of making the clustering result less sensitive to initialization and potential outliers, \cite{zadegan2013ranked} proposed the concept of hostility index based on a sum of ranks obtained from triple comparisons. 
Aside from these, \cite{achtert2006deli} proposed a dissimilarity based on the distance to the $\ell$-th nearest neighbor, which can therefore be implemented based solely on ordinal information.

\begin{quote}
{\em Besides putting ordinal $K$-medoids in the context of ordinal data, as we just did, we establish its consistency for two types of ordinal information: quadruple comparisons giving an overall ranking of all pairwise dissimilarities; and triple comparisons giving a ranking relative to each sample point.}
\end{quote}

\subsection{Setting and Content}
\label{subsec:section_1_settings_subsection}

We consider the problem of clustering some data points in a metric space into $k$ clusters, where $k$ is given. The metric space is denoted $(\cX, \dist)$ and assumed to be a locally compact Polish space. The sample is denoted $x_1, \dots, x_n$ and assumed to have been drawn from a Borel probability measure $Q$ assumed to have bounded support\footnote{This is for convenience. See \citep{parna1986strong}.} containing at least $k$ points. 
We will let $Q_n$ denote the empirical distribution, namely, $Q_n(B) := \frac1n \sum_{i=1}^n 1_{\{x_i \in B\}}$ for any set $B \subset \cX$.
For two sets $A, B \subset \cX$, define
\begin{equation}
\HD(A|B) := \sup_{a \in A} \inf_{b \in B} \dist(a, b),
\end{equation}
so that the Hausdorff distance between $A$ and $B$ is $\max\{\HD(A|B), \HD(B|A)\}$.

The organization of the paper will be as follows. 
In \secref{section_2}, we prove the asymptotic equivalence of $K$-means and $K$-medoids, and deduce from that the consistency of $K$-medoids in the metric setting.  
In \secref{ranks_application}, we consider two ordinal settings, based on quadruple and triple comparisons respectively, and establish the consistency of $K$-medoids in each case using the equivalence result from \secref{section_2}. 
We provide numerical experiments along the way to illustrate our theoretical results.
Our work is greatly inspired by that of \cite{parna1986strong}, and we will refer to his work often.

\begin{remark}
\label{canonical}
We want to mention that all our results apply when $\cX$ is a finite dimensional Banach space and $Q$ has a density with respect to the Lebesgue measure which is bounded and has compact support. 
\end{remark}

\section{Consistency of $K$-Medoids}
\label{sec:section_2}

For a $k$-tuple $A \subset \cX$, consider the risk
\begin{equation}
    \label{loss_function}
    L(A, Q) = \int_{\mathcal{X}} \min_{a \in A} \phi(\dist(x, a)) d Q(x),
\end{equation}
where $\phi : [0,\infty) \to [0,\infty)$ is a loss function assumed to be non-decreasing, continuous, and such that $\phi(d) = 0$ if and only if $d=0$ --- all these assumptions being rather standard.  
By $K$-means we mean the result of the following optimization problem:
\begin{equation}
    \label{k_means}
\text{minimize} \ L(A,Q_n) \quad
\text{over} \ A \subset \mathcal{X},\ |A| = k.
\end{equation}
And by $K$-medoids we mean the same optimization problem but restricted to $k$-tuples made of sample points:
\begin{equation}
    \label{k_medoids_version}
\text{minimize} \ L(A,Q_n) \quad
\text{over} \ A \subset \mathcal{X}_n,\ |A| = k,
\end{equation}
where $\cX_n := \{x_1, \dots, x_n\}$.
Note that 
\begin{equation}
L(A, Q_n) = \frac1n \sum_{i=1}^n \min_{a \in A} \phi(\dist(x_i, a)).
\end{equation}

It is well-known that, as formulated, \eqref{k_means} and \eqref{k_medoids_version} can behave quite differently. Take for example the case of the real line with $Q$ the uniform distribution on $[-2,-1] \cup [1,2]$. When $k = 1$, in the large-sample limit, the origin is the unique solution to $K$-means problem, while $-1$ and $1$ are the solutions to the $K$-medoids problem. Instead, we consider the following restricted form of $K$-means:
\begin{equation}
    \label{k_means_version}
\text{minimize} \ L(A,Q_n) \quad
\text{over} \ A \subset \supp(Q),\ |A| = k,
\end{equation}
where $\supp(Q)$ denotes the support of $Q$.
The analyst cannot consider this problem when the support of $Q$ is unknown, which is typically the case. But this optimization problem is only used as a device to analyze the asymptotic behavior of $K$-medoids.

\begin{theorem}
\label{theorem_1}
In the present context, $K$-medoids \eqref{k_medoids_version} is asymptotically equivalent to $K$-means \eqref{k_means_version}, which in turn is asymptotically equivalent to population version of the same problem, namely 
\begin{equation}
    \label{k_means_population}
\text{minimize} \ L(A,Q) \quad
\text{over} \ A \subset \supp(Q),\ |A| = k.
\end{equation}
We conclude that, if $A^*_n$ is a solution to \eqref{k_medoids_version}, then in probability,
\begin{equation}
L(A^*_n, Q) \xrightarrow{n\to\infty} \min_{|A| = k} L(A, Q).
\end{equation}
\end{theorem}

\begin{remark}
As discussed in \citep{cuesta1988strong, parna1990existence, parna1992clustering}, a $K$-means problem may not have a solution. In our situation, however, we are assuming that the space is a locally compact Polish space, and a solution can be shown to exist by a simple compactness argument together with our assumptions on $\phi$ (and the fact that the distance function is always continuous in any metric space it equips). This applies to \eqref{k_means}, \eqref{k_means_version} and \eqref{k_means_population}.
\end{remark}

\begin{proof}
Since everything happens within the support of $Q$, we may assume without loss of generality that $Q$ is supported on the entire space, meaning that $\supp(Q) = \cX$. And since we assume $\supp(Q)$ to be bounded, we are effectively assuming that $\cX$ is bounded, and therefore compact since it is assumed to be locally compact.

The asymptotic equivalence of \eqref{k_means_version} and \eqref{k_means_population} is the consistency result of \cite{parna1986strong}. It can be deduced easily from the arguments we present below, which themselves are by-and-large adapted from \citep{parna1986strong}. So all we are left to do is prove that \eqref{k_medoids_version} is asymptotically equivalent to  \eqref{k_means_version}. To be sure, by this we mean that, if $A_n^*$ is a solution to the former and $A_n$ a solution to the latter, then
\begin{equation}
\big|L(A_n^*, Q_n) - L(A_n, Q_n)\big| \xrightarrow{n\to\infty} 0,
\end{equation}
in probability.
Because by definition $L(A_n^*, Q_n) \ge L(A_n, Q_n)$, all we need to show is that
\begin{equation}
\limsup_{n\to\infty}\, L(A_n^*, Q_n) - L(A_n, Q_n) \le 0.
\end{equation}

The remaining of the proof consists of three steps. 
We first show in Lemma~\ref{lemma_equivalence} below that $L(A, Q_n) \to L(A, Q)$ as $n\to\infty$, uniformly over $A$. We then show in Lemma~\ref{lemma_continuous} further down that $A \mapsto L(A, Q)$ is uniformly continuous. 
The last step consists in using these results in conjunction with the `squeeze theorem'. 

By the uniform convergence established in Lemma~\ref{lemma_equivalence}, we have
\begin{equation}
\lim_{n\to\infty} |L(A^*_n, Q_n) - L(A^*_n, Q)| = 0,
\end{equation}
as well as 
\begin{equation}
\lim_{n\to\infty} |L(A_n, Q_n) - L(A_n, Q)| = 0.
\end{equation}
Therefore, all we need to show is that
\begin{equation}
\label{equivalence1}
\limsup_{n\to\infty}\, L(A_n^*, Q) - L(A_n, Q) \le 0.
\end{equation}

For every point in $A_n$ find the closest sample point, and gather all these in $B^*_n$. 
Note that by Lemma~\ref{lemma_dense}, 
\begin{equation}
h_n := \HD(A_n|B_n) = \max_{a \in A_n} \min_{b \in B^*_n} \dist(a,b) \xrightarrow{n\to\infty} 0,
\end{equation}
in probability.
Hence, by Lemma~\ref{lemma_continuous}, we have
\begin{equation}
\limsup_{n\to\infty} L(B^*_n, Q) - L(A_n, Q) \le \limsup_{n\to\infty} \omega(h_n) = 0.
\end{equation}
With the fact that $L(A_n^*, Q_n) \le L(B^*_n, Q_n)$ by definition of $A^*_n$, together with the uniform convergence also giving
\begin{equation}
\lim_{n\to\infty} |L(B^*_n, Q_n) - L(B^*_n, Q)| = 0,
\end{equation}
we thus conclude that \eqref{equivalence1} holds.
\end{proof}

\begin{lemma}
\label{lemma_dense}
Assuming that $\cX$ is compact and that $\supp(Q) = \cX$, in probability,
\begin{equation}
\HD(\cX | \cX_n) = \sup_{x \in \cX} \min_{i \in [n]} \dist(x, x_i) \to 0, \quad n \to \infty.
\end{equation}
\end{lemma}

\begin{proof}
The arguments are standard and follow from the definition of $\supp(Q)$. Indeed, $\supp(Q)$ is the complement of the largest open set $D$ in $\cX$ such that $Q(D) = 0$. Since $\supp(Q) = \cX$ by assumption, it must be that $Q(B(x, r)) > 0$ for all $x \in \cX$ and all $r > 0$, where $B(x, r)$ is defined as the closed ball centered at $x$ with radius $r$. Fix $r > 0$. Because $\cX$ is compact there is $y_1, \dots, y_m \in \cX$ such that $\cX = \bigcup_j B(y_j, r)$. By the triangle inequality,
\begin{align*}
&\HD(\cX|\cX_n) \ge 2 r \\
&\iff \exists x : \min_i \dist(x,x_i) \ge 2r \\
&\implies \exists j : \min_i \dist(y_j, x_i) \ge r,
\end{align*}
and by the union bound, this implies that
\begin{align*}
&\P(\HD(\cX|\cX_n) \ge 2 r) \\
&\le \sum_j \P(\min_i \dist(y_j, x_i) \ge r) \\
&= \sum_j (1 - Q(B(y_j, r))^n \\
&\le m (1 - \min_j Q(B(y_j, r))^n \\
&\to 0, \quad n \to \infty.
\end{align*}
Since $r>0$ is arbitrary, the claim is established.
\end{proof}

\subsection{Uniform Convergence Lemma}
\label{subsec:convergence_lemma}

\begin{lemma}
    \label{lemma_equivalence}
Assuming that $\cX$ is compact, we have, in probability,
    \begin{equation}
        \label{lemma_equality}
        \limsup_{n \to \infty} \sup_{|A| \le k} \left|L(A, Q_n) - L(A, Q)\right| = 0.
    \end{equation}
\end{lemma}

The rest of this subsection is devoted to proving this lemma. It is enough to prove the variant where $|A| \le k$ is replaced by $|A| = k$.
The proof is very similar to the proof of \citep[Lem 1]{parna1986strong}, with some differences. We provide a full proof for the sake of completeness. 

Note that, like $L(A,Q)$, $L(A,Q_n)$ can be expressed as an integral:
\begin{equation}
L(A, Q_n) = \int_{\mathcal{X}} \min_{a \in A} \phi(\dist(x, a)) d Q_n(x).
\end{equation}

To each finite set $A$, we associate the following function
\begin{equation}
    f_A(x) = \min_{a \in A} \phi(\dist(x,a)).
\end{equation}
Define the following class of functions
\begin{equation}
    \mathcal{F} = \{f_A: A \subset \mathcal{X}, \left|A\right|=k\}.
\end{equation}

\begin{lemma}[Th 3.2 of \citep{rao1962relations}]
    \label{rao_theorem_32}
Let $\mathcal{F}$ be a family of continuous functions on a separable metric space $\mathcal{X}$ which is equicontinuous and admits a continuous envelope (there is $g$ continuous such that $|f(x)| \leq g(x)$ for all $f \in \cF$). In this context, suppose that $(\mu_n)$ is a sequence of measures on $\cX$ converging weakly to $\mu$, another measure on $\cX$ with $\int g d\mu_n \xrightarrow{} \int g d\mu <\infty$. 
    Then we have:
    \begin{equation}
        \limsup_{n\xrightarrow{}\infty} \sup_{f \in \mathcal{F}} \left|\int f d\mu_n - \int f d\mu\right| = 0.
    \end{equation}
\end{lemma}

We apply this result with $\mu_n = Q_n$ and $\mu = Q$, for which the weak convergence is satisfied with probability~1 \citep{varadarajan1958convergence}. 
The existence of an envelope function $g$ satisfying the requirements for the function class of interest, $\cF$ above, is here immediate since for any $A$,
\begin{equation}
0 \le f_A(x) \le \phi(\diam(\cX)) < \infty,
\end{equation}
so that we may take $g \equiv \phi(\diam(\cX))$.
It only remains to show $\mathcal{F}$ is equicontinuous.
This amounts to showing that, for any $y_0 \in \mathcal{X}$ and any $\epsilon>0$, there exists a $\delta >0$, such that $|f_A(y_0)-f_A(y)|<\epsilon$ for any $k$-tuple $A$ and any $y \in B(y_0, \eps)$.

For the given $y_0$ and $y$, we denote $a(y_0)$ and $a(y)$ closest points in $A$ to them so that 
\begin{equation*}
    \min_{a \in A} \dist(y_0, a) - \min_{a \in A} \dist(y, a) = \dist(y_0, a(y_0)) - \dist(y, a(y)).
\end{equation*}
By definition and the triangle inequality,
\begin{align*}
&\dist(y_0, a(y_0)) - \dist(y, a(y)) \\
&\leq \dist(y_0, a(y)) - \dist(y, a(y)) \\
&\leq \dist(y_0, y),
\end{align*}
and similarly,
\begin{align*}
&\dist(y_0, a(y_0)) - \dist(y, a(y)) \\
&\geq \dist(y_0, a(y_0)) - \dist(y, a(y_0)) \\
&\geq -\dist(y_0, y).
\end{align*}
We thus deduce that
\begin{equation}
     \left|\min_{a \in A} \dist(y_0, a) - \min_{a \in A} \dist(y, a)\right| \le \dist(y_0, y).
\end{equation}

Since $\phi$ is assumed to be continuous, it is uniformly continuous on $[0,\diam(\cX)]$. Let $\omega$ denote its modulus of continuity on that interval so that
\begin{equation*}
\big|\phi(d) - \phi(d')\big|
\le \omega(|d-d'|), \quad \forall d,d' \in [0,\diam(\cX)].
\end{equation*}

We then have
\begin{align}
    &\left|f_A(y_0) - f_A(y)\right| \\ 
    &= \left|\min_{a \in A} \phi(\dist(y_0, a)) - \min_{a \in A} \phi(\dist(y, a))\right| \\
    &= \left|\phi \Big(\min_{a \in A} \dist(y_0, a)\Big) - \phi\Big(\min_{a \in A} \dist(y, a)\Big)\right|\\
    &\le \omega(\dist(y_0, y)),
\end{align}
using the monotonicity of $\phi$ along the way.
We have proved that $\cF$ is indeed equicontinuous.
Therefore the proof of Lemma~\ref{lemma_equivalence} is complete.

\subsection{Uniform Continuity Lemma}
\label{subsec:continuity_lemma}

\begin{lemma}
    \label{lemma_continuous}
For any two sets $A, B \subset \cX$, we have
\begin{equation}
L(B, Q) \le L(A, Q) + \omega(\HD(A|B)),
\end{equation}
where $\omega$ is the modulus of continuity of $\phi$ on $[0, \diam(\cX)]$.
\end{lemma}

The rest of this subsection is devoted to proving this lemma.
Fix two sets $A, B \subset \mathcal{X}$, and let $h := \HD(A|B)$. For any $a \in A$, define $b_a$ as the closest point in $B$ to $a$. 
Notice that by definition:
\begin{equation}
    \dist(a, b_a) \leq h,
\end{equation} 
and thus with the triangle inequality, for any point $x$ we have:
\begin{equation}
    \label{d_x_a_larger}
    \dist(x, a) \geq \dist(x, b_a) - \dist(a, b_a) \geq \dist(x, b_a) - h.
\end{equation}
Taking minimums we get:
\begin{equation}
    \label{minimum_a_distance}
    \min_{a \in A} \dist(x, a)  \geq \min_{a \in A} \dist(x, b_a) - h \geq \min_{b \in B} \dist(x, b) - h.
\end{equation}

Using the fact that $\phi$ is non-decreasing, we then have:
\begin{align}
    &\min_{a \in A} \phi(\dist(x, a)) - \min_{b \in B} \phi(\dist(x, b)) \\
    &= \phi\Big(\min_{a \in A} \dist(x, a)\Big) - \phi\Big(\min_{b \in B} \dist(x, b)\Big) \\
    &\geq -\omega (h). 
\end{align}
Therefore, by integrating with respect to $Q$, we obtain:
\begin{align*}
    &L(A, Q) - L(B, Q) \\
    &= \int \min_{a \in A} \phi(\dist(x, a)) dQ(x) - \int \min_{b \in B} \phi(\dist(x, b)) dQ(x) \\
    &= \int \Big[\min_{a \in A} \phi(\dist(x, a)) - \min_{b \in B} \phi(\dist(x, b))\Big] dQ(x) \\
    &\geq -\omega(h).
\end{align*}

\subsection{Simulations}

We report on a simple experiment illustrating the asymptotic equivalence established in Theorem~\ref{theorem_1}. To keep a balance between the necessity to probe an asymptotic result ($n$ large enough) and computational feasibility ($n$ not too large), we choose to work with a sample of size $n=2000$. We generate data from two equally weighted Gaussian distributions in $R^2$, centered at $(-0.5,0)$ and $(0.5,0)$, each with covariance $0.05 \times {I_2}$. Each setting is repeated $50$ times. The result of this experiments is summarized in Table~\ref{tab:equivalence_2cluster}. As can be seen from this experiment, although varying according to different metrics and loss functions, the performance of $K$-means and $K$-medoids are indeed very similar.

\begin{table}[h]
\centering
\caption{Mean values and standard deviations of the Average Center Error (error) and the Adjusted Rand Index (ARI) of $K$-means and $K$-medoids for various metrics and loss functions.}
\label{tab:equivalence_2cluster}
\bigskip
\setlength{\tabcolsep}{0.05in}
\begin{tabular}{p{0.06\textwidth} p{0.12\textwidth} p{0.12\textwidth} p{0.12\textwidth} }
\toprule
\text{   } & \text{   } & {\bf $K$-means} & {\bf $K$-medoids} \\ \midrule

$L_1$ & error [$\times 10^{-2}$] & 1.2 (0.4) & 1.8 (0.7) \\
\text{   } & ARI & 0.780 (0.017) & 0.778 (0.016)\\ \midrule

$\sqrt{L_2}$ & error [$\times 10^{-2}$] & 8.9 (1.7) & 11.7 (2.5) \\ 
\text{   } & ARI & 0.784 (0.020) & 0.782 (0.022)\\ \midrule

$L_2$ & error [$\times 10^{-3}$] & 9.4 (3.2) & 12.3 (4.3) \\ \text{   } & ARI & 0.789 (0.017) & 0.789 (0.017)\\ \midrule

${L_2}^2$ & error [$\times 10^{-4}$] & 1.1 (0.9) & 2.3 (1.7) \\ 
\text{   } &ARI & 0.785 (0.016) & 0.785 (0.016)\\ \midrule
 
$L_\infty$ & error [$\times 10^{-3}$] & 8.9 (3.4) & 12.0 (4.2) \\ 
\text{   } & ARI & 0.785 (0.021) & 0.783 (0.020)\\ 
\bottomrule
\end{tabular}
\end{table}

\section{Consistency of Ordinal $K$-Medoids}
\label{sec:ranks_application}

In this section we consider the problem of clustering with only an ordering of the dissimilarities. We consider two such orderings, one based on quadruple comparisons and another based on triple comparisons. We apply the results from \secref{section_2} to show that, in both cases, $K$-medoids is consistent.

\subsection{Quadruple Comparisons}
\label{subsec:3_1}
First we consider a situation in which all quadruple comparisons of the form `Is $\dist(x_i, x_j)$ larger or smaller than $\dist(x_l, x_m)$?' are available. Equivalently, this is a situation in which a complete ordering of the pairwise dissimilarities is available.

For $i \in [n]$, let $R_i(a)$ denote the rank of $\dist(x_i, a)$ among $\{\dist(x_l, x_m): l < m\}$, and for a $k$-tuple $A$, define 
\begin{align}
S_{\rm rank}(A, Q_n) = \frac{1}{n} \sum_{i=1}^n \min_{a \in A} \frac{R_i(a)}{\frac{n(n-1)}{2}}.
\end{align}
By ordinal $K$-medoids we mean the following optimization problem:
\begin{equation}
\label{k_medoids_quadruple}
\text{minimize} \ S_{\rm rank}(A,Q_n) \quad
\text{over} \ A \subset \mathcal{X}_n,\ |A| = k.
\end{equation}
This problem can be posed with the available information, and thus in principle can be solved.
The equivalent restricted variant of ordinal $K$-means corresponds the following optimization problem:
\begin{equation}
\text{minimize} \ S_{\rm rank}(A,Q_n) \quad
\text{over} \ A \subset \supp(Q),\ |A| = k.
\end{equation}
As before, the latter is used as a bridge to show that the former is asymptotically equivalent to the population version of this $K$-means problem, which is given by
\begin{equation}
\text{minimize} \ S(A,Q) \quad
\text{over} \ A \subset \supp(Q),\ |A| = k,
\end{equation}
where 
\begin{align}
\label{S_population}
S(A, Q) := \int \min _{a \in A} G(\dist(x, a)) dQ(x),
\end{align}
with
\begin{equation}
    \label{gd_definition}
    G(t) := \P(\dist(X, X^\prime)\leq t),
\end{equation}
$X, X^\prime$ being independent with distribution $Q$. 

Here is the missing link between $S_{\rm rank}$ and $S$.
\begin{lemma}
\label{lemma_S}
The following holds in probability:
\begin{equation}
\limsup_{n\to\infty} \sup_{|A| \le k} |S_{\rm rank}(A, Q_n) - S(A, Q_n)| = 0.
\end{equation}
\end{lemma}

\begin{proof}
Let $\hat G_n$ denote the empirical distribution function of all the pairwise distances between sample points, meaning, 
\[\hat G_n(t) := \frac2{n(n-1)} \sum_{l < m} 1_{\{\dist(x_l,x_m) \le t\}}.\]
By the law of large numbers for $U$-statistics, in probability, $\hat G_n(t) \to G(t)$ as $n \to \infty$ for every fixed $t$. The Glivenko--Cantelli lemma does not quite apply as the pairwise distances are not an iid sample, but the two ingredients are there \citep{van2000asymptotic}: pointwise convergence as just stated, and the fact that $\hat G_N$ and $G$ are both distribution functions in that they both are non-decreasing from 0 to 1 on $[0,\infty)$. Hence, 
\[\eps_n := \sup_t |\hat G_n(t) - G(t)| \xrightarrow{n \to\infty} 0,\]
in probability.
We then have:
\begin{align*}
R_i(a)
&= \sum_{l < m} 1_{\{\dist(x_l,x_m) \le \dist(x_i, a)\}} \\
&= \frac{n(n-1)}2 \hat G_n(\dist(x_i, a)) \\
&= \frac{n(n-1)}2 G(\dist(x_i, a)) \pm \frac{n(n-1)}2 \eps_n,
\end{align*}
giving 
\begin{align*}
S_{\rm rank}(A, Q_n)
&= \frac1n \sum_{i=1}^n \min_{a \in A} G(\dist(x_i, a)) \pm \eps_n \\
&= S(A, Q_n) \pm \eps_n,
\end{align*}
for any finite $A$, which establishes the result.
\end{proof}

Establishing the consistency of ordinal $K$-medoids is now a straightforward consequence of Theorem~\ref{theorem_1}.
We need to assume that $G$ defined above is continuous, which is the case in the canonical situation of Remark~\ref{canonical}.

\begin{theorem}
    \label{theorem_2}
In the present context, if $A^*_n$ is a solution to ordinal $K$-medoids in the form of \eqref{k_medoids_quadruple}, then in probability,
\begin{equation}
S(A^*_n, Q) \xrightarrow{n\to\infty} \min_{|A| = k} S(A, Q).
\end{equation}
\end{theorem}

\begin{proof}
By Lemma \ref{lemma_S}, we have that ordinal $K$-medoids \eqref{k_medoids_quadruple} is asymptotically equivalent to the following problem:
\begin{equation}
\text{minimize} \ S(A,Q_n) \quad
\text{over} \ A \subset \cX_n,\ |A| = k.
\end{equation}
But $S$ is exactly as $L$ in Section~\ref{sec:section_2}, with $G$ replacing $\phi$ there, and since $G$ satisfies the same properties assumed of $\phi$, Theorem~\ref{theorem_1} applies to yield the claim.
\end{proof}

\subsection{Triple Comparisons}
\label{subsec:3_2}
We turn to a situation in which only triple comparisons of the form `Is $\dist(x_i, x_j)$ larger or smaller than $\dist(x_i, x_l)$?' are available. We do assume that all of these comparisons are on hand. Equivalently, this is a situation in which an ordering of the pairwise dissimilarities involving a particular point are available.

Hence, we work here with the ranks (re)defined as follows.
For $i \in [n]$, let $R_i(a)$ denote the rank of $\dist(x_i, a)$ among $\{\dist(x_i, x_j): j \ne i\}$, and for a $k$-tuple $A$, define 
\begin{align}
S_{\rm rank}(A, Q_n) = \frac{1}{n} \sum_{i=1}^n \min_{a \in A} \frac{R_i(a)}{n-1}.
\end{align}
Ordinal $K$-medoids and (restricted) ordinal $K$-means are otherwise defined as before.
The population equivalent to ordinal $K$-means is now given by
\begin{equation}
\text{minimize} \ S(A,Q) \quad
\text{over} \ A \subset \supp(Q),\ |A| = k,
\end{equation}
where 
\begin{align}
\label{S_population_triple}
S(A, Q) := \int \min _{a \in A} G^x(\dist(x, a)) dQ(x),
\end{align}
with
\begin{equation}
    \label{gd_definition}
    G^x(t) := \P(\dist(x, X^\prime)\leq t),
\end{equation}
$X^\prime$ having distribution $Q$. 

\begin{lemma}
\label{lemma_S_triple}
The following holds in probability:
\begin{equation}
\limsup_{n\to\infty} \sup_{|A| \le k} |S_{\rm rank}(A, Q_n) - S(A, Q_n)| = 0.
\end{equation}
\end{lemma}

\begin{proof}
Define
\[\hat G_{n,i}(t) := \frac1{n-1} \sum_{j \ne i} 1_{\{\dist(x_i,x_j) \le t\}}.\]
By the Dvoretzky–Kiefer–Wolfowitz (DKW) inequality, for each $i$ and any $\eps>0$, we have:
\begin{equation*}
\P\Big( \sup_t |\hat G_{n,i}(t) - G^{x_i}(t)| > \epsilon\Big) \leq 2\exp(-2(n-1)\epsilon^2).
\end{equation*}
With this, and the union bound, we obtain: 
\[\eps_n := \max_i \sup_t |\hat G_{n,i}(t) - G^{x_i}(t)| \xrightarrow{n \to\infty} 0,\]
in probability.
We then have:
\begin{align*}
R_i(a)
&= \sum_{j \ne i} 1_{\{\dist(x_i,x_j) \le \dist(x_i, a)\}} \\
&= (n-1) \hat G_{n,i}(\dist(x_i, a)) \\
&= (n-1) G^{x_i}(\dist(x_i, a)) \pm (n-1) \eps_n,
\end{align*}
giving 
\begin{align*}
S_{\rm rank}(A, Q_n)
&= \frac1n \sum_{i=1}^n \min_{a \in A} G^{x_i}(\dist(x_i, a)) \pm \eps_n \\
&= S(A, Q_n) \pm \eps_n,
\end{align*}
for any finite $A$, which establishes the result.
\end{proof}

The following is our consistency result for $K$-medoids based on triple comparisons. It is not an immediate consequence of Theorem~\ref{theorem_1}, but the proof arguments are parallel. We need to make additional assumption that $(x,t) \mapsto Q(B(x,t))$ is continuous on $\cX \times (0,\infty)$. 
This is the case in the canonical situation of Remark~\ref{canonical}.

\begin{theorem}
    \label{theorem_3}
In the present context, if $A^*_n$ is a solution to ordinal $K$-medoids based on triple comparisons, then in probability,
\begin{equation}
S(A^*_n, Q) \xrightarrow{n\to\infty} \min_{|A| = k} S(A, Q),
\end{equation}
now with $S$ defined as in \eqref{S_population_triple}.
\end{theorem}

\begin{proof}
By Lemma \ref{lemma_S_triple}, we have that ordinal $K$-medoids is asymptotically equivalent to the following problem:
\begin{equation}
\text{minimize} \ S(A,Q_n) \quad
\text{over} \ A \subset \cX_n,\ |A| = k.
\end{equation}
But unlike the situation in Theorem~\ref{theorem_2}, now $S$ is {\em not} exactly as $L$ in Section~\ref{sec:section_2}, complicating matters a little bit. Nevertheless, the proof arguments are parallel to those underlying Theorem~\ref{theorem_1}.
As we did there, we need to establish uniform convergence and uniform continuity. As before, we may assume without loss of generality that $\cX$ is compact and that $\supp(Q) = \cX$. In that case, $(x,t) \mapsto Q(B(x,t))$ is uniformly continuous, and we let $\Omega$ denote its modulus of continuity so that 
\begin{equation*}
\big|Q(B(x,s)) - Q(B(y,t))\big|
\le \Omega(\dist(x,y), |s-t|),
\end{equation*}
for all $x,y \in \cX$ and all $s,t  > 0$.

For the uniform convergence, the proof of Lemma~\ref{lemma_equivalence} proceeds as before until the very end where instead
\begin{align}
    &\left|f_A(y_0) - f_A(y)\right| \\ 
    &= \left|\min_{a \in A} G^{y_0}(\dist(y_0, a)) - \min_{a \in A} G^y(\dist(y, a))\right| \\
    &= \left|G^{y_0}(\dist(y_0, A)) - G^y(\dist(y, A))\right| \\
    &\le \Omega(\dist(y_0, y), |\dist(y_0,A) - \dist(y,A)|) \\
    &\le \Omega(\dist(y_0, y), \dist(y_0, y)) \\
    &\to 0, \quad \text{when } \dist(y_0, y) \to 0.
\end{align}

For the uniform continuity, the proof of Lemma~\ref{lemma_continuous} proceeds as before except that
\begin{align}
    &\min_{a \in A} G^x(\dist(x, a)) - \min_{b \in B} G^x(\dist(x, b)) \\
    &= G^x(\dist(x, A)) - G^x(\dist(x, B)) \\
    &\geq -\Omega(0, h), 
\end{align}
with $h := \HD(A|B)$ as in that proof, so that the statement of that lemma continues to hold but with $\omega(t) := \Omega(0,t)$.
\end{proof}

\subsection{Simulations}
\label{subsec:3_simulation}

We again report on a numerical experiment showcasing the results derived in this section in the context of ordinal clustering. We chose to work with a sample of size $n=750$. We generate data from three equally weighted Gaussian distributions in two dimensions, centered at $(-0.5,0)$, $(0.5,0)$ and $(0, \sqrt{3}/2)$, each with covariance $0.05 \times I_2$. Each setting is repeated $50$ times. The result of our experiment is summarized in Table~\ref{tab:ranks_3_clusters}. As can be seen from this experiment, $K$-medoids based on ordinal information performs nearly as well as $K$-medoids based on the full dissimilarity information.

\begin{table}
\centering
\caption{Mean values and standard deviations of the Average Center Error (error) and the Adjusted Rand Index (ARI) for various metrics and loss functions for $K$-medoids based on triple-comparisons (TC), quadruple-comparisons (QC), and the actual distances (KM).}
\centering
\label{tab:ranks_3_clusters}
\bigskip
\setlength{\tabcolsep}{0.05in}
\begin{tabular}{p{0.04\textwidth} p{0.12\textwidth}  p{0.08\textwidth} p{0.08\textwidth} p{0.08\textwidth} }
\toprule
\text{   } & \text{   } & {\bf TC} & {\bf QC} & {\bf KM} \\ \midrule
$L_1$ & error [$\times 10^{-2}$] & 4.4 & 3.6 & 3.7  \\
\text{   } & \text{   } & (1.3) & (1.3) & (1.2) \\
\text{   } & ARI & 0.924 & 0.924 & 0.925\\ 
\text{   } & \text{   } & (0.014)& (0.014) & (0.014)\\ 
\midrule

$\sqrt{L_2}$ & error [$\times 10^{-1}$] & 1.8 & 1.6 & 1.7  \\
\text{   } & \text{   } & (0.3) &  (0.2) & (0.3) \\
\text{   } & ARI & 0.928 & 0.929 & 0.929\\ 
\text{   } & \text{   } & (0.016)& (0.015) & (0.016)\\ 
\midrule

$L_2$ & error [$\times 10^{-2}$] & 3.2 & 2.7 & 2.7  \\
\text{   } & \text{   } & (0.8) &  (0.9) & (0.9) \\
\text{   } & ARI & 0.934 & 0.933 & 0.933\\ 
\text{   } & \text{   } & (0.016)& (0.016) & (0.016)\\ 
\midrule

${L_2}^2$ & error [$\times 10^{-3}$] & 1.4 &  0.9 & 0.8  \\
\text{   } & \text{   } & (0.7) &  (0.5) & (0.5) \\
\text{   } & ARI & 0.931 & 0.931 & 0.930\\ 
\text{   } & \text{   } & (0.020)& (0.019) & (0.020)\\ 
\midrule

$L_\infty$ & error [$\times 10^{-2}$] & 3.0 &  2.6 & 2.7  \\
\text{   } & \text{   } & (1.1) &  (1.0) & (1.0) \\
\text{   } & ARI & 0.918 & 0.917 & 0.918\\ 
\text{   } & \text{   } & (0.017)& (0.017) & (0.017)\\ 
\bottomrule
\end{tabular}
\end{table}

\section{Discussion}

In this paper, we have shown the asymptotic equivalence of $K$-means and $K$-medoids, and used this equivalence to prove the consistency of $K$-medoids in metric and non-metric situations.

\subsection{Consistency of the Solution}
Our consistency results are on the value of the optimization problem defining $K$-medoids in the various settings we considered. Specifically, we showed in each case that $T(A^*_n, Q) \to_{n\to\infty} \min_A T(A, Q)$, in probability, where $T$ is an appropriate criterion (either $L$ or one of the two variants of $S$) and $A^*_n$ is the solution to $K$-medoids. What about the behavior of the solution $A^*_n$ itself?

Here the situation is completely generic: if the solution to the population problem, namely $A_{\rm opt} := \argmin_A T(A, Q)$, is unique, then $A^*_n \to_{n\to\infty} A_{\rm opt}$, again in probability.
This is simply due to the fact that in our setting we can reduce the situation to when $\cX$ is compact, and in all cases we considered $A \mapsto T(A, Q)$ is continuous.

\subsection{Clustering After Embedding?}
It might be possible to establish the consistency of ordinal $K$-medoids building on the consistency of ordinal embedding. This route appears unnecessarily sophisticated, however, in particular in light of a more straightforward approach that we built on the work of \cite{parna1986strong}. And from a computational standpoint, performing $K$-medoids in the ordinal setting has essentially the same complexity as in the regular (i.e., metric) setting, while methods for ordinal embedding tend to be much more demanding in computational resources.  
 
\subsection{A `Bad' Variant of $K$-Medoids} 
In the setting where triple comparisons are available, instead of defining the ranks as we did, we could have worked with the following definition. 
For $i \in [n]$, let $R_i(a)$ denote the rank of $\dist(x_i, a)$ among $\{\dist(x_j, a): j \in [n]\}$.
Although the resulting method can be analyzed in very much the same way, it turns out to not be useful for the purpose of clustering.
This is due to the fact that the corresponding optimization problem accepts a large range of solutions. To see this, consider the case $k=1$. With the corresponding definition of $S_{\rm rank}$, we have that 
\begin{equation}
S_{\rm rank}(a,Q_n) = \frac{1 + 2 + \dots + n}{n(n-1)},
\end{equation}
for all $a \in \{x_1,...,x_n\}$.
And the problem persists for other values of $k$.
For another example, consider clustering points distributed uniformly between $[-1,1]$ into $k=2$ clusters. It is clear that the correct population centers for $K$-means here are $\{-1/2, 1/2\}$. However, it can be seen that for any $1/2 \leq c \leq 1$, $A = \{-c, c\}$ also achieves the optimal population risk. 

\section*{Acknowledgements}
We are very grateful to Professor Kalev P\"arna for sharing with us his papers, which we could not otherwise access. 
This work was partially supported by the National Science Foundation (DMS 1916071).  

\bibliographystyle{chicago}
\bibliography{robust.bib}

\begin{thebibliography}{}

\bibitem[\protect\citeauthoryear{Achtert, B{\"o}hm, and Kr{\"o}ger}{Achtert
  et~al.}{2006}]{achtert2006deli}
Achtert, E., C.~B{\"o}hm, and P.~Kr{\"o}ger (2006).
\newblock Deli-clu: boosting robustness, completeness, usability, and
  efficiency of hierarchical clustering by a closest pair ranking.
\newblock In {\em Pacific-Asia Conference on Knowledge Discovery and Data
  Mining}, pp.\  119--128. Springer.

\bibitem[\protect\citeauthoryear{Arias-Castro}{Arias-Castro}{2017}]{arias2017some}
Arias-Castro, E. (2017).
\newblock Some theory for ordinal embedding.
\newblock {\em Bernoulli\/}~{\em 23\/}(3), 1663--1693.

\bibitem[\protect\citeauthoryear{Bradley and Terry}{Bradley and
  Terry}{1952}]{bradley1952rank}
Bradley, R.~A. and M.~E. Terry (1952).
\newblock Rank analysis of incomplete block designs: I. the method of paired
  comparisons.
\newblock {\em Biometrika\/}~{\em 39\/}(3/4), 324--345.

\bibitem[\protect\citeauthoryear{Cuesta and Matr{\'a}n}{Cuesta and
  Matr{\'a}n}{1988}]{cuesta1988strong}
Cuesta, J. and C.~Matr{\'a}n (1988).
\newblock The strong law of large numbers for {$K$}-means and best possible
  nets of {B}anach valued random variables.
\newblock {\em Probability Theory and Related Fields\/}~{\em 78\/}(4),
  523--534.

\bibitem[\protect\citeauthoryear{Ester, Kriegel, Sander, and Xu}{Ester
  et~al.}{1996}]{ester1996density}
Ester, M., H.-P. Kriegel, J.~Sander, and X.~Xu (1996).
\newblock A density-based algorithm for discovering clusters in large spatial
  databases with noise.
\newblock In {\em Kdd}, Volume~96, pp.\  226--231.

\bibitem[\protect\citeauthoryear{H{\'a}jek and Sid{\'a}k}{H{\'a}jek and
  Sid{\'a}k}{1967}]{hajek1967theory}
H{\'a}jek, J. and Z.~Sid{\'a}k (1967).
\newblock Theory of rank tests.

\bibitem[\protect\citeauthoryear{Herrndorf}{Herrndorf}{1983}]{herrndorf1983approximation}
Herrndorf, N. (1983).
\newblock Approximation of vector-valued random variables by constants.
\newblock {\em Journal of Approximation Theory\/}~{\em 37\/}(2), 175--181.

\bibitem[\protect\citeauthoryear{Huang, Wang, and Zhu}{Huang
  et~al.}{2020}]{huang2020adaptive}
Huang, T., S.~Wang, and W.~Zhu (2020).
\newblock An adaptive kernelized rank-order distance for clustering
  non-spherical data with high noise.
\newblock {\em International Journal of Machine Learning and Cybernetics\/},
  1--13.

\bibitem[\protect\citeauthoryear{Jain, Murty, and Flynn}{Jain
  et~al.}{1999}]{jain1999data}
Jain, A.~K., M.~N. Murty, and P.~J. Flynn (1999).
\newblock Data clustering: a review.
\newblock {\em ACM Computing Surveys (CSUR)\/}~{\em 31\/}(3), 264--323.

\bibitem[\protect\citeauthoryear{Kaufman and Rousseeuw}{Kaufman and
  Rousseeuw}{1987}]{kaufmann1987clustering}
Kaufman, L. and P.~Rousseeuw (1987).
\newblock Clustering by means of medoids.
\newblock In {\em Statistical Data Analysis Based on the $L_1$ Norm Conference,
  Neuchatel, 1987}, pp.\  405--416.

\bibitem[\protect\citeauthoryear{Kaufman and Rousseeuw}{Kaufman and
  Rousseeuw}{1990}]{kaufman725pj}
Kaufman, L. and P.~Rousseeuw (1990).
\newblock Finding groups in data: An introduction to cluster analysis.
\newblock {\em Hoboken NJ John Wiley \& Sons Inc\/}~{\em 725}.

\bibitem[\protect\citeauthoryear{Kaufman and Rousseeuw}{Kaufman and
  Rousseeuw}{2009}]{kaufman2009finding}
Kaufman, L. and P.~Rousseeuw (2009).
\newblock {\em Finding Groups in Data: An Introduction to Cluster Analysis},
  Volume 344.
\newblock John Wiley \& Sons.

\bibitem[\protect\citeauthoryear{Kleindessner and Luxburg}{Kleindessner and
  Luxburg}{2014}]{kleindessner2014uniqueness}
Kleindessner, M. and U.~Luxburg (2014).
\newblock Uniqueness of ordinal embedding.
\newblock In {\em Conference on Learning Theory}, pp.\  40--67.

\bibitem[\protect\citeauthoryear{Kruskal}{Kruskal}{1964}]{kruskal1964multidimensional}
Kruskal, J.~B. (1964).
\newblock Multidimensional scaling by optimizing goodness of fit to a nonmetric
  hypothesis.
\newblock {\em Psychometrika\/}~{\em 29\/}(1), 1--27.

\bibitem[\protect\citeauthoryear{Lloyd}{Lloyd}{1982}]{lloyd1982least}
Lloyd, S. (1982).
\newblock Least squares quantization in pcm.
\newblock {\em IEEE transactions on information theory\/}~{\em 28\/}(2),
  129--137.

\bibitem[\protect\citeauthoryear{MacQueen}{MacQueen}{1967}]{macqueen1967some}
MacQueen, J. (1967).
\newblock Some methods for classification and analysis of multivariate
  observations.
\newblock In {\em Fifth Berkeley Symposium on Mathematical Statistics and
  Probability}, Volume~1, pp.\  281--297. Oakland, CA, USA.

\bibitem[\protect\citeauthoryear{Park and Jun}{Park and
  Jun}{2009}]{park2009simple}
Park, H.-S. and C.-H. Jun (2009).
\newblock A simple and fast algorithm for {$K$}-medoids clustering.
\newblock {\em Expert systems with applications\/}~{\em 36\/}(2), 3336--3341.

\bibitem[\protect\citeauthoryear{P{\"a}rna}{P{\"a}rna}{1986}]{parna1986strong}
P{\"a}rna, K. (1986).
\newblock Strong consistency of ${K}$-means clustering criterion in separable
  metric spaces.
\newblock {\em Tartu Riikl. Ul. Toimetised\/}~{\em 733}, 86--96.

\bibitem[\protect\citeauthoryear{P{\"a}rna}{P{\"a}rna}{1988}]{parna1988stability}
P{\"a}rna, K. (1988).
\newblock On the stability of {$K$}-means clustering in metric spaces.
\newblock {\em Tartu Riikl. Ul. Toimetised\/}~{\em 798}, 19--36.

\bibitem[\protect\citeauthoryear{P{\"a}rna}{P{\"a}rna}{1990}]{parna1990existence}
P{\"a}rna, K. (1990).
\newblock On the existence and weak convergence of {$K$}-centres in {B}anach
  spaces.
\newblock {\em Tartu Ulikooli Toimetised\/}~{\em 893}, 17--287.

\bibitem[\protect\citeauthoryear{P{\"a}rna}{P{\"a}rna}{1992}]{parna1992clustering}
P{\"a}rna, K. (1992).
\newblock Clustering in metric spaces: some existence and continuity results
  for {$K$}-centers.
\newblock In {\em Analyzing and Modeling Data and Knowledge}, pp.\  85--91.
  Springer.

\bibitem[\protect\citeauthoryear{Pollard}{Pollard}{1981}]{pollard1981strong}
Pollard, D. (1981).
\newblock Strong consistency of ${K}$-means clustering.
\newblock {\em The Annals of Statistics\/}, 135--140.

\bibitem[\protect\citeauthoryear{Rao}{Rao}{1962}]{rao1962relations}
Rao, R.~R. (1962).
\newblock Relations between weak and uniform convergence of measures with
  applications.
\newblock {\em The Annals of Mathematical Statistics\/}, 659--680.

\bibitem[\protect\citeauthoryear{Shepard}{Shepard}{1962a}]{shepard1962i}
Shepard, R.~N. (1962a).
\newblock The analysis of proximities: multidimensional scaling with an unknown
  distance function. i.
\newblock {\em Psychometrika\/}~{\em 27\/}(2), 125--140.

\bibitem[\protect\citeauthoryear{Shepard}{Shepard}{1962b}]{shepard1962ii}
Shepard, R.~N. (1962b).
\newblock The analysis of proximities: Multidimensional scaling with an unknown
  distance function. ii.
\newblock {\em Psychometrika\/}~{\em 27\/}(3), 219--246.

\bibitem[\protect\citeauthoryear{Shepard}{Shepard}{1966}]{shepard1966metric}
Shepard, R.~N. (1966).
\newblock Metric structures in ordinal data.
\newblock {\em Journal of Mathematical Psychology\/}~{\em 3\/}(2), 287--315.

\bibitem[\protect\citeauthoryear{Van Der~Laan, Pollard, and Bryan}{Van Der~Laan
  et~al.}{2003}]{van2003new}
Van Der~Laan, M., K.~Pollard, and J.~Bryan (2003).
\newblock A new partitioning around medoids algorithm.
\newblock {\em Journal of Statistical Computation and Simulation\/}~{\em
  73\/}(8), 575--584.

\bibitem[\protect\citeauthoryear{Van Der~Vaart}{Van
  Der~Vaart}{1998}]{van2000asymptotic}
Van Der~Vaart, A.~W. (1998).
\newblock {\em Asymptotic statistics}.
\newblock Cambridge University Press.

\bibitem[\protect\citeauthoryear{Varadarajan}{Varadarajan}{1958}]{varadarajan1958convergence}
Varadarajan, V.~S. (1958).
\newblock On the convergence of sample probability distributions.
\newblock {\em Sankhy{\=a}: The Indian Journal of Statistics
  (1933-1960)\/}~{\em 19\/}(1/2), 23--26.

\bibitem[\protect\citeauthoryear{Wang, Li, Bucci, Liang, Chen, and
  Varshney}{Wang et~al.}{2019}]{wang2019k}
Wang, T., Q.~Li, D.~J. Bucci, Y.~Liang, B.~Chen, and P.~K. Varshney (2019).
\newblock {$K$}-medoids clustering of data sequences with composite
  distributions.
\newblock {\em IEEE Transactions on Signal Processing\/}~{\em 67\/}(8),
  2093--2106.

\bibitem[\protect\citeauthoryear{Zadegan, Mirzaie, and Sadoughi}{Zadegan
  et~al.}{2013}]{zadegan2013ranked}
Zadegan, S. M.~R., M.~Mirzaie, and F.~Sadoughi (2013).
\newblock Ranked {$K$}-medoids: A fast and accurate rank-based partitioning
  algorithm for clustering large datasets.
\newblock {\em Knowledge-Based Systems\/}~{\em 39}, 133--143.

\bibitem[\protect\citeauthoryear{Zhu, Wen, and Sun}{Zhu
  et~al.}{2011}]{zhu2011rank}
Zhu, C., F.~Wen, and J.~Sun (2011).
\newblock A rank-order distance based clustering algorithm for face tagging.
\newblock In {\em CVPR 2011}, pp.\  481--488. IEEE.

\end{thebibliography}

\end{document}